\documentclass[preprint, 12pt, authoryear]{elsarticle}

\usepackage{epstopdf}
\usepackage[english]{babel}
\usepackage{babel}
\usepackage{fontenc}
\usepackage{graphicx}
\usepackage{amsmath,amsthm,amssymb}
\usepackage[colorlinks=true,citecolor=blue]{hyperref}
\bibpunct{[}{]}{,}{n}{,}{,}
\usepackage{subfigure}
\usepackage{color}
\usepackage{newlfont}
\usepackage{array}
\usepackage{bigints}
\usepackage{multirow}
\usepackage{multicol}
\usepackage{caption}
\usepackage{subcaption}
\usepackage{tabularx}
\usepackage{longtable} 
\usepackage{ifthen}
\usepackage{alltt}
\usepackage{enumitem}
\usepackage[text={6.25in,8.5in},centering]{geometry}
\usepackage{tikz}
\usepackage{pdflscape}
\usepackage{float}
\usepackage{xparse}
\usepackage{orcidlink}
\usepackage{lineno}
\usepackage{epstopdf}
\usetikzlibrary{shapes,arrows}

\renewcommand{\arraystretch}{2}
\NewDocumentCommand{\evalat}{sO{\big}mm}{%
	\IfBooleanTF{#1}
	{\mleft. #3 \mright|_{#4}}
	{#3#2|_{#4}}%
}

\newcommand{\ba}{\begin{array} }
\newcommand{\ea}{\end{array} }
\newcommand{\bae}{\begin{eqnarray}}
\newcommand{\eae}{\end{eqnarray}}
\newcommand{\bea}{\begin{eqnarray*}}
\newcommand{\eea}{\end{eqnarray*}}
\newcommand{\be}{\begin{equation}}
\newcommand{\ee}{\end{equation}}

\newcommand{\f}{\displaystyle\frac}

\newcommand{\ld}{\left\{ }
\newcommand{\rd}{\right\} }

\newcommand{\lx}{\left(}
\newcommand{\rx}{\right)}
\newcommand{\lz}{\left[ }
\newcommand{\rz}{\right] }

\def\to{{\rightarrow}}

\def\a{{\alpha}}
\def\b{{\beta}}

\def\d{{\delta}}

\def\g{{\gamma}}

\def\la{{\lambda}}

\def\th{{\theta}}

\def\to{{\rightarrow}}

\numberwithin{equation}{section}

\newtheorem{proposition}{\bf Proposition}[section]

\input epsf.sty
\journal{Applied Mathematics and Computations}

\begin{document}

\begin{frontmatter}


\title{Addendum to "Persistence and extinction in an Elk-Wolf prey-predator system with refuge and inter-regional movement. Appl. Math. Comput. 514 (2026) 129834"}

\author[inst1]{Rajesh Das}
\ead{rajesh_d@amsc.iitr.ac.in}

\author[inst2]{Dibakar Ghosh}
\ead{dibakar@isical.ac.in}

\author[inst1]{Sourav Kumar Sasmal \corref{cor1}}
\ead{sourav.sasmal@amsc.iitr.ac.in}


\cortext[cor1]{Corresponding author.}

\address[inst1]{Department of Applied Mathematics and Scientific Computing, Indian Institute of Technology Roorkee, Roorkee, Uttarakhand - 247667, India}

\address[inst2]{Physics and Applied Mathematics Unit, Indian Statistical Institute, 203, B. T. Road, Kolkata - 700108, India}

\begin{abstract}
The elk–wolf model with movements between refuge and open habitat was put forward in Maji et al. \cite{maji2026persistence}, which is rigorously re-examined in this remark. We re-evaluate the local and global stability analyses, especially the construction of the Lyapunov function, and provide mathematical clarifications on boundedness, model formulation, and the existence of equilibria. The sensitivity and numerical results are re-examined for consistency and re-producibility, and the Hopf bifurcation conditions are re-derived using the proper transversality criteria. The purpose of this note is to support future studies of predator-prey systems based on refuges by offering mathematically consistent conditions.
\end{abstract}

\end{frontmatter}


\section{Introduction}\label{sec_introduction}

Resource–consumer interactions are central to ecological stability and species coexistence, with the elk–wolf system serving as a classical predator–prey example \cite{hebblewhite2013consequences, trump2022sustainable}. Prey species often reduce predation risk by seeking refuge in safer habitats. Banff National Park in Canada acts as a managed refuge for elk from the nearby Bow Valley, where elk–wolf interactions are frequent. Movement of elk between refuge and non-refuge areas can significantly influence system dynamics \cite{goldberg2014consequences}. This study formulates and analyzes a mathematical model to examine these refuge-mediated elk–wolf interactions. \\

Goldberg et al. \cite{goldberg2014consequences} analyzed five distinct elk–wolf mathematical models under different ecological scenarios relevant to Banff National Park and the adjacent Bow Valley region. Their study indicated that the model, which separates the Banff townsite elk population from that of the Bow Valley, while representing elk–wolf interactions in the Bow Valley through a classical Lotka–Volterra type predator–prey framework, provided the most appropriate description for the ecological setting considered. However, elk–wolf interactions also occur within the Banff townsite area. Incorporating this additional ecological realism, Maji et al. \cite{maji2026persistence} formulated and investigated the elk–wolf system
\begin{eqnarray}\label{system_AMC-2025}
    \frac{dE}{dt}&=& \a E\lx 1-\f{E}{K}\rx-\g EP-q\psi E,\nonumber\\
    \frac{dN}{dt}&=& \b N+\mu E-\xi NP,\\
    \frac{dP}{dt}&=& \theta_1\g EP+\theta_2\xi NP-\eta P,\nonumber
\end{eqnarray}
with the initial conditions: $E(t=0)\ge 0,\ N(t=0)\ge 0,\ \text{and}\ P(t=0)\ge 0.$ There they divided the elk population into two subpopulations. $E$ represents the elk portion that stays at the Banff townsite area, whereas $N$ denotes the elk portion that lives in the Bow Valley wild area. $P$ represents the predator wolves that prey on both subpopulations, but the encounter rate at the Banff townsite region ($\g$) is very low compared to the Bow Valley region ($\xi$). $\a$ represents the growth (birth - death - elk dispersal from Banff townsite to Bow Valley region) rate of $E$ population without predation, whereas $\b$ represents the growth (birth - death) rate of $N$ population without predation. $\eta$ represents the death rate of wolf ($P$) population. $\th_1,\ \th_2$ are the biomass conversion efficiencies to the wolf population from elk populations of the Banff area and the Bow Valley area, respectively. $\mu$ is the natural elk movement rate from the Banff area to the Bow Valley region. $q$ and $\psi$ are the relocation rate and relocation effort. They assumed linear functional responses for species interactions. The default values of these parameters for numerical simulation are given in Table \ref{table_parameters}. \\

They established that system \eqref{system_AMC-2025} is bounded and positively invariant within the biologically feasible region. The existence of all possible equilibria was determined, and their local stability was analyzed using the Routh–Hurwitz criteria. Furthermore, by constructing an appropriate Lyapunov function, sufficient conditions were derived for the global stability of the coexistence equilibrium. The study also demonstrated the occurrence of the Hopf bifurcation in the vicinity of the unique coexistence equilibrium, leading to sustained oscillatory dynamics in the $N$–$P$ plane, as the $E$ population eventually becomes extinct beyond the bifurcation threshold. They also determined the stability and direction of the bifurcating limit cycle using normal form theory as developed in Hassard et al. \cite{hassard1981theory} and Kuznetsov et al. \cite{ kuznetsov1998elements}. In the numerical study, the parameters $q,\ \psi,\ \b,\ \mu,\ \xi,$ and $\eta$ were estimated via a least-squares nonlinear fitting procedure implemented in Mathematica. Furthermore, PRCC analysis was conducted over a $100\%$ variation around baseline values using $400$ time points to examine parameter sensitivity. \\

In this comment, we reassess several analytical and numerical results of Maji et al. \cite{maji2026persistence}. The parameter set used therein does not appear to satisfy the boundedness conditions of system \eqref{system_AMC-2025}. We identify inaccuracies in the expressions for the coexistence equilibrium, its local stability analysis, the construction of the Lyapunov function, and the resulting global stability conditions. Certain inconsistencies also arise in the Hopf bifurcation criteria and in the derivation of the stability and direction of the bifurcating limit cycle. Furthermore, the reported Hopf bifurcation around the coexistence equilibrium seems to be misinterpreted. If the $E$ population becomes extinct beyond a threshold and the reduced $N$–$P$ subsystem exhibits oscillations, this corresponds to a branch point bifurcation rather than a Hopf bifurcation. Indeed, once $E$ vanishes, the $N$–$P$ subsystem reduces to a Lotka–Volterra model, which admits oscillatory solutions without invoking a Hopf mechanism. The PRCC results and the stability regions in Fig.~7 of \cite{maji2026persistence} also require clarification, as the coexistence equilibrium does not exist throughout a substantial portion of the region indicated therein. 

\begin{table}[]
    \centering\small
    \setlength{\tabcolsep}{6pt}
    \renewcommand{\arraystretch}{1.1}
    \begin{tabular}{clcc}
       \hline
       Parameter  & Description & Values & Source\\
       \hline
       $\a$ & Banff townsite elk growth rate  & $0.25$ & \cite{goldberg2014consequences}\\
       $K$  & Carrying capacity of Banff townsite elk & $1000$ &\cite{goldberg2014consequences}\\
       $\g$ & Banff townsite elk capture rate of wolves  & $0.05$& Assumed \\
       $q$ & Relocation/harvesting rate of the Banff townsite elk & $0.02$ & Assumed\\
       $\psi$ & Relocation/harvesting effort & $0.01$ & \cite{maji2026persistence}\\
       $\b$ & Bow Valley elk growth rate & $0.16$ & \cite{goldberg2014consequences}\\
       $\mu$ & Elk movement rate from Banff area to Bow Valley & $0.10$ & Assumed\\
       $\xi$ & Bow Valley elk capture rate of wolves  & $0.10$ & Assumed\\
       $\th_1$ & Conversion efficiencies of wolves for Banff townsite elk & $0.001$ & Assumed\\
       $\th_2$ & Conversion efficiencies of wolves for Bow Valley elk  & $0.01$ & Assumed\\
       $\eta$ & Death rate of wolves & $0.30$ & \cite{goldberg2014consequences}\\
       \hline
    \end{tabular}
    \caption{Default parameter values and their definitions used in system \eqref{system_AMC-2025}.}
    \label{table_parameters}
\end{table}


\section{Mathematical Analysis}\label{sec_mathematical analysis} 

\subsection{Positivity and boundedness of the system}\label{subsec_boundedness}
\begin{proposition}
    The system \eqref{system_AMC-2025} is positively invariant and bounded if $0<\b<\eta-\th_1\g K.$
\end{proposition}

\begin{proof}
For $E=0,\ N=0,$ and $P=0$, we have, $\f{dE}{dt}\ge 0,\ \f{dN}{dt}\ge 0,$ and $\f{dP}{dt}\ge 0.$ Then for any initial conditions from the first octant that is, $\lx E(0),\ N(0),\ P(0)\rx \in \mathbb{R}^3_+$, the solutions to system \eqref{system_AMC-2025} will remain within the first octant. Therefore, system \eqref{system_AMC-2025} is positively invariant in $\mathbb{R}^3_+$. \\

For the boundedness of system \eqref{system_AMC-2025}, we have
\begin{eqnarray*}
    \f{dE}{dt} \le  \a E\lx 1-\f{E}{K}\rx \implies E_{max} \le  K.
\end{eqnarray*}
We also have
\begin{eqnarray*}
    &\f{dP}{dt}& = P\lx \theta_1\g E+\theta_2\xi N-\eta\rx,\\
    && \le P\lx \theta_1\g K+\theta_2\xi N-\eta\rx. 
\end{eqnarray*}
We assume, $N^\#=\f{\eta-\theta_1\g K}{\theta_2\xi}>0$, provided $\eta>\theta_1\g K$. 
Now, there are two subcases, namely, $N<N^\#$ and $N>N^\#.$\\
\textbf{Case-I:} $N<N^\#$\\
Easily explainable that $\f{dP}{dt}<0\implies P(t)<P\ (t=0).$ Eventually, all the populations are bounded with the bounds: $E\le K,\ N<N^\#,$ and $P<P~(t=0).$ \\
\textbf{Case-II:} $N>N^\#$\\
Consider, $X(t)=\theta_2N(t)+P(t)$, which implies:
\begin{eqnarray*}
    \f{dX}{dt} &=& \th_2\b N+\th_2\mu E-\th_2\xi NP +\th_1\g EP+\th_2\xi NP-\eta P,\\
    &=& \th_2\b N+\th_2\mu E +\th_1\g EP-\eta P,\\
    &\le & \th_2\b N+\th_2\mu K-P(\eta-\th_1\g K),\\
    &\le & \th_2\mu K-\lx \eta-\th_1\g K-\b\rx X.
\end{eqnarray*}
Since $\th_2N(t)\le X(t),$ and $P(t)\le X(t)$ for all non negative $N(t),\ P(t).$
We write, $\f{dX}{dt}\le A-BX$, where $A=\th_2\mu K$, and $B=\eta -\th_1\g K-\b$. If $B>0$, that is if $\eta-\th_1\g K>\b>0$ then $\limsup_{t\to \infty}{X(t)}= \f{A}{B}$.\\
Thus, the solutions to system \eqref{system_AMC-2025} are bounded with $E(t)\le K, \ N(t)\le \f{\mu K}{\eta-\th_1\g K-\b},$ and $P(t)\le \f{\th_2\mu K}{\eta-\th_1\g K-\b}$. 
\end{proof}

\subsection{Equilibria of the system}\label{subsec_equilibria}
The system \eqref{system_AMC-2025} has two boundary equilibria, namely, the extinction equilibrium $X_0=\lx0,\ 0,\ 0 \rx$, the equilibrium without Banff townsite elk, $X_1=\lx 0,\ \f{\eta}{\th_2\xi},\ \f{\b}{\xi} \rx$, and the unique coexistence equilibrium $X^*=\lx E^*,\ N^*,\ P^*\rx$, where the expression of the densities at $X^*$ are given in \eqref{exp_coexistence}. \\

For the coexistence equilibria of system \eqref{system_AMC-2025}, we have the following set of equations 
\begin{eqnarray}\label{eq for coexistence}
   \a \lx 1 - \f{E^*}{K}\rx - \g P^* - q\psi  &=0,\nonumber\\
   \b N^* + \mu E^*-\xi N^*P^* &=0,\\
   \theta_1\g E^* + \theta_2\xi N^* - \eta  &=0.\nonumber
\end{eqnarray}

If $\f{\eta}{\th_2\xi}-\f{K\th_1\g\lx \a-q\psi\rx}{\a \th_2\xi}<N^*<\f{\eta}{\th_2\xi},$ solving system of equations \eqref{eq for coexistence}, the unique coexistence equilibrium $X^*$ of the system \eqref{system_AMC-2025} is given as $X^*=(E^*,\ N^*,\ P^*),$ where
\begin{eqnarray}\label{exp_coexistence}
    &N^*& = \f{m_1+\sqrt{m_1^2+4m_2}}{2},\nonumber\\
    &E^*& = \f{\eta-\th_2\xi N^*}{\th_1\g},\\
    &P^*& = \f{\th_1\g K\lx \a-q\psi\rx-\a\lx \eta-\th_2\xi N^*\rx}{\th_1\g^2 K}\nonumber,
\end{eqnarray}
and $m_1=\f{\a \eta\xi+\th_1\g K\lx \b\g+\xi q \psi-a\xi\rx - \th_2\g K\mu\xi}{\a\th_2\xi^2}$, $m_2=\f{K\mu\eta\g}{\a\th_2\xi^2}.$\\

Table \ref{table_existence_stability} summarizes the conditions for the existence and the local stability associated with the equilibria of system \eqref{system_AMC-2025}. 

\begin{table}[]
    \centering\small
    \setlength{\tabcolsep}{6pt}
    \renewcommand{\arraystretch}{1.1}
    \begin{tabular}{lll}
    \hline
        Equilibrium & Existence Conditions & Local Stability Conditions  \\
        \hline
        $X_0$ & Always exists & Always unstable\\
        $X_1$ & Always exists & $\a-q\psi<\f{\b \g}{\xi}$ (Spirally stable) \\
        $X^*$ & $\f{\eta}{\th_2\xi}-\f{\th_1\g K\lx \a-q\psi\rx}{\a \th_2\xi}<N^*<\f{\eta}{\th_2\xi}$ & $N^*>\max{\ld N_1,\ N_2\rd}$ and $F(N^*)>0$ \\ 
        \hline
    \end{tabular}
    \caption{The local stability and existence conditions of equilibria of system \eqref{system_AMC-2025}, where $N_1,\  N_2$, and $F(N^*)$ are given in \eqref{exp_N1N2FN}.}
    \label{table_existence_stability}
\end{table}
\begin{eqnarray}\label{exp_N1N2FN}
    & N_1 = & \f{\eta}{\th_2\xi}-\f{\th_1\g K}{\a(\xi-\g)\th_2\xi}\ld \xi(\a-q\psi)-\b\g\rd ,\nonumber\\
    & N_2 = & \f{\eta}{2\th_2\xi}-\f{\mu\g K}{2\a\xi}-\f{\th_1\g K}{2\a\th_2\xi}\ld \xi(\a-q\psi)-\b\g\rd,\\
    & F(N^*)= & A_{11}{N^*}^3 + A_{22}{N^*}^2+ A_{33}{N^*}+ A_{44},\nonumber
\end{eqnarray}
where,
\begin{eqnarray*}
    & A_{11} = & a_1a_3+\th_1\g^3 Ka_6,\\
    & A_{22} = & a_1a_4+a_2a_3+\th_1\g^3 Ka_7,\\
    & A_{33} = & a_1a_5+a_2a_4+\th_1\g^3 Ka_8,\\
    & A_{44} = & a_2a_5+\th_1\g^3 Ka_9,
\end{eqnarray*}
and
\begin{eqnarray*}
    & a_1 = & -\a \th_2 \xi^2, \\
    & a_2 = & \a \eta \xi + \th_1\g K\ld \b \g 
      - \xi (\a - q \psi)\rd, \\
    & a_3 = & \a^2 \th^2_2 \xi^2 (\xi - \g) 
      - \a \th_1\th_2^2 \g\xi^3 K, \\
    & a_4 = & \a \th_1\th_2\g\eta \xi^2 K 
      - \th_1^2\th_2 \g^2 \xi^2K^2 (\a - q \psi) 
      - \th_1\th_2\a \g \xi K\ld \b \g - \xi (\a - q \psi) \rd
      - 2 \th_2 \a^2 \xi \eta (\xi - \g), \\
    & a_5 = & \a^2 \eta^2 (\xi - \g) + \th_1 \g K \a \eta \ld \b \g - \xi (\a - q \psi)\rd, \\
    & a_6 = & \th_2^3 \xi^3\a^2, \\
    & a_7 = & \a K \th_2^2 \xi^2 \ld \th_1 \g (\a - q \g) - \th_2 \mu \xi \rd 
      - 3 \eta \a^2 \th_2^2 \xi^2, \\
    & a_8 = & 3 \th_2^2 \a^2 \xi \eta^2 + \th_1\th_2^2\xi^2 \mu\g K^2 (\a - q \psi) 
      - 2 \th_2\xi\a K \eta \ld \th_1\g(\a - q \psi)-\th_2\mu\xi\rd, \\
    & a_9 = & -\a^3 \eta^3 + \a K \eta^2 \ld \th_1 \g (\a - q \psi) - \th_2\mu\xi \rd
      - \th_1 \th_2 \xi \mu \eta \g K^2 (\a - q \psi).
\end{eqnarray*}
\subsection{Global stability of the coexistence equilibrium}\label{subsec_global stability}
We consider the following Lyapunov function to prove the global stability of the coexisting equilibrium $X^*$:
\begin{equation}\nonumber
    W\lx t\rx=\delta_1\lx E(t)-E^*-E^*\ln\f{E(t)}{E^*}\rx+\f{\delta_2}{2}\Big(N(t)-N^*\Big)^2+\delta_3\lx P(t)-P^*-P^*\ln\f{P(t)}{P^*}\rx,
\end{equation}
where $\delta_i>0$ for $i=1,2,3$. After differentiating with respect to time $t$ and grouping terms using the system \eqref{system_AMC-2025} and equations \eqref{eq for coexistence}, we have 
\begin{eqnarray*}
    W'(t)=-Q(x,y)-\g (\d_1-\theta_1\d_3)xz-\xi(\d_2N-\theta_1\d_3)yz,
\end{eqnarray*}
where $x=E-E^*,\ y=N-N^*,$ and $z=P-P^*$. $Q(x,y)$ is a quadratic function of $x$ and $y$. If we consider $\d_1>\theta_1\d_3$ and $\d_2N_{min}>\theta_2\d_3$ with $Q(x,y)$ being positive definite, then it implies the negativity of $W'(t)$. 

As we have $Q(x,y)= Ax^2-2Bxy+Cy^2$, with $A=\f{\a \d_1}{K},\ B=\f{\mu \d_2}{2}$, and $C=\f{\mu\d_2E^*}{N^*}$. The positive definiteness of $Q(x,y)$ is guaranteed with the condition $AC-B^2>0$, i.e., $4\a \d_1E^*>\mu\d_2KN^*.$

\subsection{Existence of Hopf bifurcation}
The coexistence equilibrium will be changing its stability through the Hopf bifurcation with respect to $\b$ at $\b^\#$ if the following conditions hold:
\begin{itemize}
    \item $\forall\  i=1,2,3,\ b_i(\b^\#)>0,$
    \item $b_1(\b^\#)b_2(\b^\#)-b_3(\b^\#)=0,$ and 
    \item $\f{d}{d\b}\lz Re(\la(\b))\rz|_{\b=\b^\#}\ne0,$
\end{itemize}
where 
\begin{equation}\label{eq_characteristic}
    \la^3+b_1\la^2+b_2\la+b_3=0,
\end{equation}
is the characteristic equation of the Jacobian evaluated at the unique coexistence equilibrium of system \eqref{system_AMC-2025}. Here,
\begin{align*}
    b_1&= \f{\a E^*}{K}-\b+\xi P^*,\\
    b_2&= -\f{\a \b E^*}{K}+\th_1\g^2E^*P^*+\th_2\xi^2N^*P^*+\f{\a \xi E^*P^*}{K},\\
    b_3&= E^*P^*\ld \th_1\g^2 (-\b+\xi P^*) + \th_2\g\mu\xi +\f{\th_2\a\xi^2N^*}{K}\rd.
\end{align*}

Now, at $\b=\b^\#,$ we have $\la_{1,2}=\pm i\sqrt{b_2}$, and $\la_3=-b_1,$ where $i=\sqrt{-1}.$ Then, for any $\epsilon>0$, $\b\in(\b^\#-\epsilon,\ \b^\#+\epsilon)$ then assume $\la_{1,2}=\phi_1\pm i\phi_2$. Thus, substituting these in the values of $\la_1$ in equation \eqref{eq_characteristic}, and separating the real and imaginary parts, we obtain

\begin{eqnarray}\label{eq_phi_1'}
    P_1\phi_1'-P_2\phi_2'+R_1=0,\nonumber\\
    P_2\phi_1'+P_1\phi_2'+R_2=0,
\end{eqnarray}
where,
\begin{eqnarray}\label{exp_P1P2R1R2}
    &P_1=&3(\phi_1^2-\phi_2^2)+2b_1\phi_1+b_2,\nonumber\\
    &P_2=&2b_1\phi_2+6\phi_1\phi_2,\nonumber\\
    &R_1=&b_1'(\phi_1^2-\phi_2^2)+b_2'\phi_1+b_3',\nonumber\\
    &R_2=&2b_1'\phi_1\phi_2+b_2'\phi_2.\nonumber
\end{eqnarray}
Thus, at $\b=\b^\#,$ we have the following
\begin{eqnarray}\label{exp_beta*}
    &P_1(\b^\#)=&-2b_2(\b^\#),\nonumber\\
    &P_2(\b^\#)=&2b_1(\b^\#)\sqrt{b_2(\b^\#)},\nonumber\\
    &R_1(\b^\#)=&-b_1'(\b^\#)b_2(\b^\#)+b_3'(\b^\#)\\
    &R_2(\b^\#)=&b_2'(\b^\#)\sqrt{b_2(\b^\#)}.\nonumber
\end{eqnarray}
Then solving for $\phi_1'$ from equations \eqref{eq_phi_1'}, and using expressions \eqref{exp_beta*}, we obtain

\begin{eqnarray*}
    &\phi_1'&=-\f{P_1(\b^\#)R_1(\b^\#)+P_2(\b^\#)R_2(\b^\#)}{P_1^2(\b^\#)+P_2^2(\b^\#)},\\
    & &=-\f{-2b_2(\b^\#)\ld -b_1'(\b^\#)b_2(\b^\#)+b_3'(\b^\#)\rd+2b_1(\b^\#)b_2(\b^\#)b_2'(\b^\#)}{4b_2^2(\b^\#)+4b_1^2(\b^\#)b_2(\b^\#)},\\
    & &= \f{b_3'(\b^\#)-b_1(\b^\#)b_2'(\b^\#)-b_1'(\b^\#)b_2(\b^\#)}{2\ld b_1^2(\b^\#)+b_2(\b^\#)\rd},\ \text{since,}\ b_2(\b^\#)>0,\\
    & & =\f{1}{2(b_1^2(\b^\#)+b_2(\b^\#))}\left.\f{d}{d\b}\ld b_3(\b)-b_1(\b)b_2(\b)\rd \right|_{\b=\b^\#}.
\end{eqnarray*}
Now, since $b_i(\b^\#)>0$ for $i=1,2,3,$ we have the condition for the Hopf bifurcation to occur is
\begin{equation}\nonumber
   \left.\f{d}{d\b}\ld b_3(\b)-b_1(\b)b_2(\b)\rd \right|_{\b=\b^\#}\ne0.
\end{equation}

\begin{figure}[H]
    \centering
    \includegraphics[width=0.9\linewidth, height=0.6\linewidth]{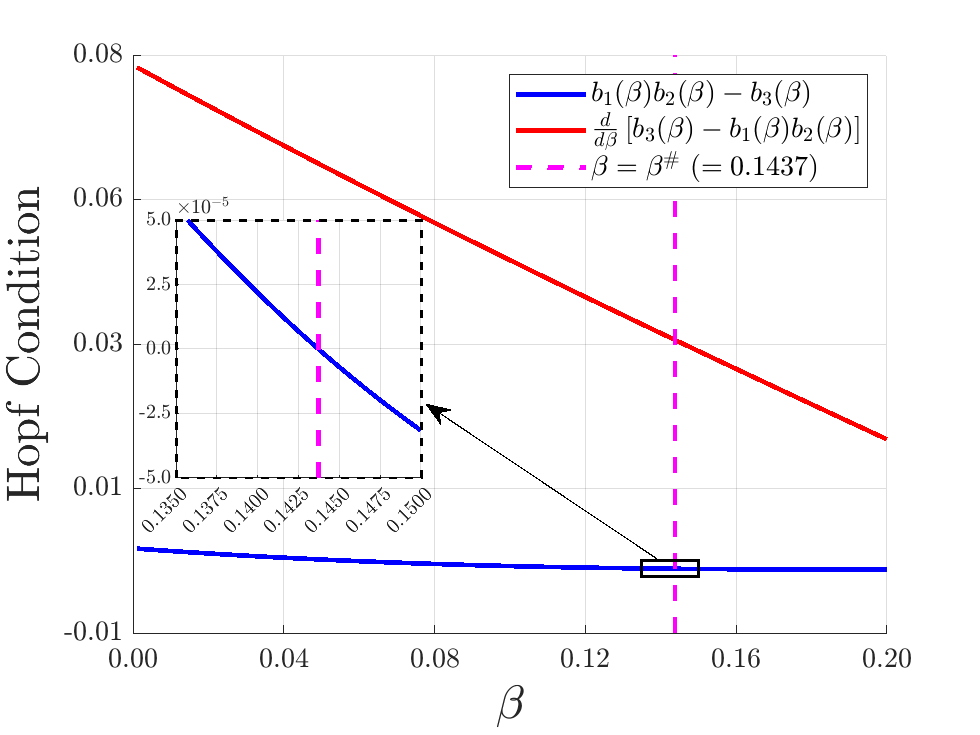}
    \caption{Figure shows the existence and transversality conditions for occurance of the Hopf bifurcation around the unique coexistence equilibrium $X^*$ of system \eqref{system_AMC-2025} with respect to bifurcation parameter $\b$, when $\g\ (=0.11)>\xi\ (=0.10),$ and the other parameters are as in Table \ref{table_parameters}. At $\b=\b^\#\ (=0.1437)$, the coexistence equilibrium satisfies the existence and transversality condition of Hopf bifurcation. That is, $b_1(\b^\#)b_2(\b^\#)-b_3(\b^\#)=0,$ and $\left.\f{d}{d\b}\ld b_3(\b)-b_1(\b)b_2(\b)\rd \right|_{\b=\b^\#}\approx 0.03\ (\ne0)$.}
    \label{fig_Hopf existence and transversality condition}
\end{figure}

\subsection{Direction, stability of limit cycle}

For the stability and direction of the bifurcating limit cycle, the variational matrix of system \eqref{system_AMC-2025} at the coexistence equilibrium $X^*$,
\begin{equation}\label{Jacobian at coexistence}\nonumber
    J_{X^*}= {\renewcommand{\arraystretch}{1.0}\begin{bmatrix}
        B_{11} & 0&B_{13}\\
        B_{21} & B_{22} & B_{23}\\
        B_{31} & B_{32} & 0
    \end{bmatrix}},   
\end{equation}
where $B_{11}=-\f{\a E^*}{K},\ B_{13}=-\g E^*,\ B_{21}=\mu,\ B_{22}=-\xi P^*, \ B_{23}=-\xi N^*,\ B_{31}=\th_1\g P^*,$ and $B_{32}=\th_2\xi P^*$. The eigenvectors of the matrix $J_{X^*}$ corresponding to the eigenvalues $\lambda_1=i\Psi_0,\ (\Psi_0=\sqrt{b_2})$ and $\lambda_3=-b_1$, respectively are given as
\begin{equation}\label{exp_U V vectors}
    U={\renewcommand{\arraystretch}{1.0}\begin{bmatrix}
        c_{11}+ic_{12}\\
        c_{21}+ic_{22}\\
        c_{31}+ic_{32}
    \end{bmatrix}},\  \text{and}\ V={\renewcommand{\arraystretch}{1.0}\begin{bmatrix}
        c_{13}\\
        c_{23}\\
        c_{33}
    \end{bmatrix}}, 
\end{equation}
where 
$c_{11}=1,\ c_{12}=0,\ c_{13}=1,\ c_{21}=-\f{B_{13}B_{31}+\Psi_0^2}{B_{13}B_{32}},\ c_{22}=-\f{B_{11}\Psi_0}{B_{13}B_{32}}, \ c_{23}=\f{b_1(B_{11}+b_1)-B_{13}B_{31}}{B_{13}B_{32}},\ $
$c_{31}=-\f{B_{11}}{B_{13}},\ c_{32}=\f{\Psi_0}{B_{13}},$ and $c_{33}=-\f{B_{11}+b_1}{B_{13}},$ provided $B_{11}B_{13}B_{31}+B_{13}B_{21}B_{32}+B_{22}B_{23}B_{32}=B_{11}B_{22}(B_{11}+B_{22}),$ and $\Psi_0^2=B_{11}B_{22}-B_{13}B_{31}-B_{23}B_{32}$.\\

Vectors $U$ and $V$ are normal to each other if the following two conditions hold:
\begin{eqnarray*}
    && \lx \Psi_0^2 +B_{13}B_{31} \rx \ld B_{13}B_{31} - b_1(B_{11}+b_1) \rd+B_{11}B_{32}^2(B_{11}+b_1)=0, \text{and}\\
    && B_{11}\ld b_1(B_{11}+b_1) - B_{13}B_{31} \rd +B_{32}^2(B_{11}+b_1)=0.
\end{eqnarray*}

Now, consider the following transformations: 
\begin{eqnarray}\label{transformations}
    &E=& E^* + x + z,\nonumber\\
    &N=& N^* + c_{21}x + c_{22}y + c_{23}z,\nonumber\\
    &P=& P^* + c_{31}x + c_{32}y + c_{33}z,\nonumber
\end{eqnarray}
where, $c_{ij},\ \forall \ i,j=1,2,3$ are given in \eqref{exp_U V vectors}, with $c_{11}=c_{13}=1$ and $c_{12}=0$. Applying the above transformation to system \eqref{system_AMC-2025}, and differentiating with respect to time $t$, we will have the following transformed system \eqref{system_transformed}.
\begin{align}\label{system_transformed}
    &\f{dx}{dt}=   \f{1}{|C|}\ld (c_{22}c_{33}-c_{23}c_{32})H_1 + c_{32}H_2 - c_{22}H_3\rd &=M_1,\nonumber\\
    &\f{dy}{dt}=  \f{1}{|C|}\ld (c_{23}c_{31}-c_{21}c_{33})H_1 + (c_{33}-c_{31})H_2 + (c_{21}-c_{23})H_3\rd &=M_2,\\
    &\f{dz}{dt}=  \f{1}{|C|}\ld (c_{21}c_{32}-c_{22}c_{31})H_1 - c_{32}H_2 + c_{22}H_3\rd &=M_3,\nonumber
\end{align}
where, 
\begin{eqnarray}\label{determinant_C matrix}
    & |C| &=det\lx{\renewcommand{\arraystretch}{1.0}\begin{bmatrix}
        1 & 0 & 1\\
        c_{21} & c_{22} & c_{23}\\
        c_{31} & c_{32} & c_{33}  
    \end{bmatrix}}\rx,\nonumber\\
    & &= c_{22}(c_{33} - c_{31}) +c_{32}(c_{21} - c_{23}),
\end{eqnarray}
and $H_i\ \text{for}\ i=1,2,3$ are the right-hand side of system \eqref{system_AMC-2025} when we put the transformed values of $E,\ N, \ P$ as given in \eqref{transformations}. That is,
\begin{eqnarray}\label{exp_H1 H2 H3}
    & H_1 = & \a\lx E^* + x + z \rx \ld \lx 1-\f{E^* + x + z}{K}\rx -\f{\g}{\a} \lx P^*+c_{31}x+c_{32}y+c_{33}z\rx - \f{q\psi}{\a} \rd, \nonumber\\
    & H_2 = & \lx N^*+c_{21}x+c_{22}y+c_{23}z\rx \ld \b - \xi \lx P^*+c_{31}x+c_{32}y+c_{33}z\rx \rd+\mu\lx E^* + x + z\rx ,\nonumber\\
    & H_3 = & \lx P^*+c_{31}x+c_{32}y+c_{33}z\rx \ld \th_1\g \lx E^* + x + z\rx  + \th_2\xi \lx N^*+c_{21}x+c_{22}y+c_{23}z\rx -\eta \rd.\nonumber
\end{eqnarray}

Clearly, $(0,\ 0,\ 0)$ is an equilibrium of transformed system \eqref{system_transformed}. The Jacobian of system \eqref{system_transformed} will be 
\begin{equation}\label{jacobian_transformed system}
    J_X={\renewcommand{\arraystretch}{1.1}\begin{bmatrix}
        \f{\partial M_1}{\partial x} &  \f{\partial M_1}{\partial y} &  \f{\partial M_1}{\partial z}\\
        \f{\partial M_2}{\partial x} &  \f{\partial M_2}{\partial y} &  \f{\partial M_2}{\partial z}\\
        \f{\partial M_3}{\partial x} &  \f{\partial M_3}{\partial y} &  \f{\partial M_3}{\partial z}  
    \end{bmatrix}},\nonumber
\end{equation}
with the conditions $\f{\partial M_1}{\partial x}=\f{\partial M_1}{\partial z}=\f{\partial M_2}{\partial y}=\f{\partial M_2}{\partial z}=\f{\partial M_3}{\partial x}=\f{\partial M_3}{\partial y}=0,\ \f{\partial M_2}{\partial x}=-\f{\partial M_1}{\partial y}=\Psi_0,$ and $\f{\partial M_3}{\partial z}=D_1.$ 

Now, we will calculate 
\begin{align}\label{exp_g20 g11 g02}
g_{20} &= \frac{1}{4} \lz \lx \f{\partial^2 M_1}{\partial x^2}
- \f{\partial^2 M_1}{\partial y^2} + 2 \f{\partial^2 M_2}{\partial x \partial y}\rx + i \lx\f{\partial^2 M_2}{\partial x^2} - \f{\partial^2 M_2}{\partial y^2} - 2 \f{\partial^2 M_1}{\partial x \partial y}\rx \rz,\nonumber \\
g_{11} &= \frac{1}{4} \lz \lx \f{\partial^2 M_1}{\partial x^2} + \f{\partial^2 M_2}{\partial y^2}\rx + i \lx\f{\partial^2 M_2}{\partial x^2} + \f{\partial^2 M_1}{\partial y^2}\rx \rz, \nonumber\\
g_{02} &= \f{1}{4} \lz \left(\f{\partial^2 M_1}{\partial x^2}
- \f{\partial^2 M_1}{\partial y^2} - 2 \f{\partial^2 M_2}{\partial x \partial y}\right) + i \lx\f{\partial^2 M_2}{\partial x^2} - \f{\partial^2 M_2}{\partial y^2} + 2  \f{\partial^2 M_1}{\partial x \partial y}\rx \rz.\nonumber
\end{align}

Further,
\begin{equation}\label{exp_g21}\nonumber
g_{21} = G_{21} + 2 G_{110} w_{11} + G_{101} w_{20},
\end{equation}
where,
\begin{align}\label{exp_G21 G110 G101}
   G_{21}=& \f{1}{8} \lz \lx \f{\partial^3 M_1}{\partial x^3} + \f{\partial^3 M_1}{\partial x \partial y^2} + \f{\partial^3 M_2}{\partial x^2 \partial y} + \f{\partial^3 M_2}{\partial y^3} \rx + i \lx \f{\partial^3 M_2}{\partial x^3} + \f{\partial^3 M_2}{\partial x \partial y^2} - \f{\partial^3 M_1}{\partial x^2 \partial y} - \f{\partial^3 M_1}{\partial y^3} \rx \rz,\nonumber\\
   G_{110} = & \f{1}{2} \lz \lx \f{\partial^2 M_1}{\partial x \partial z} + \f{\partial^2 M_2}{\partial y \partial z} \rx + i \lx \f{\partial^2 M_2}{\partial x \partial z} - \f{\partial^2 M_1}{\partial y \partial z} \rx\rz, \nonumber\\
   G_{101} = & \f{1}{2} \lz \lx \f{\partial^2 M_1}{\partial x \partial z} - \f{\partial^2 M_2}{\partial y \partial z} \rx + i \lx \f{\partial^2 M_2}{\partial x \partial z} + \f{\partial^2 M_1}{\partial y \partial z} \rx \rz,\nonumber
\end{align}
and $w_{11},\ w_{20}$ are calculated from the relations $D_1 w_{11} = - h_{11},$ and $(|C| - 2 i \Psi_0) w_{20} = - h_{20},$ where $D_1=\f{\partial M_3}{\partial z}$, $M_3$ is given in \eqref{system_transformed}, $|C|$ is given in \eqref{determinant_C matrix}, and

\begin{eqnarray}\label{exp_h11 h20}
    && h_{11} = \f{1}{4} \lx \f{\partial^2 M_3}{\partial x^2} + \f{\partial^2 M_3}{\partial y^2} \rx,\nonumber\\
    && h_{20} = \f{1}{4} \lx \f{\partial^2 M_3}{\partial x^2} - \f{\partial^2 M_3}{\partial y^2} - 2i \f{\partial^2 M_3}{\partial x \partial y} \rx.\nonumber
\end{eqnarray}

Then the first Lyapunov coefficient is given by
\begin{equation}\label{exp_first lyapunov coefficient}\nonumber
l_1(0) = \f{i}{2\Psi_0} \lx g_{20} g_{11} - 2|g_{11}|^2 - \f{1}{3} |g_{02}|^2 \rx + \f{1}{2} g_{21}.
\end{equation}

The other required components are
\begin{align}\label{exp_other coefficients}
S_1 &= -\f{\mathrm{Re}\{l_1(0)\}}{p'(0)}, \nonumber\\
S_2 &= 2\,\mathrm{Re}\{l_1(0)\}, \\
S_3 &= -\f{\mathrm{Im}\{l_1(0)\} + S_1 q'(0)}{\Psi_0},\nonumber
\end{align}

where $p'(0) = \left.\f{d}{d\b} \mathrm{Re}(\lambda(\b))\right|_{\b=\b^\#},$ and $q'(0)= \left.\f{d}{d\b} \mathrm{Im}(\lambda(\b))\right|_{\b=\b^\#}.$\\

The coefficients given in \eqref{exp_other coefficients} have different significances regarding the Hopf bifurcation. The coefficient $S_1$ provides the direction of the Hopf bifurcation. The Hopf bifurcation is subcritical for $S_1<0$ while supercritical for $S_1>0$. The coefficient $S_2$ provides the necessary information about the periodic solution's stability. If $S_2<0$, periodic solutions are stable, while the periodic solutions are unstable if $S_2>0$. The coefficient $S_3$ is associated with the period of the periodic solutions that occur. The period decreases when $S_3<0$ and increases with $S_3>0$.

\section{Numerical Simulation}\label{sec_Numerical simulation}

This section provides an extensive numerical simulations that verify the theoretical findings. The default parameter values are listed in Table \ref{table_parameters}. Some of the parameter values fall within the $95\%$ credible intervals reported in Goldberg et al. \cite{goldberg2014consequences}, while the remaining parameters are chosen hypothetically for the purpose of analysis and based on satisfying the boundedness conditions of system \eqref{system_AMC-2025}.

\begin{figure}[H]
    \centering
    \includegraphics[width=1.0\linewidth]{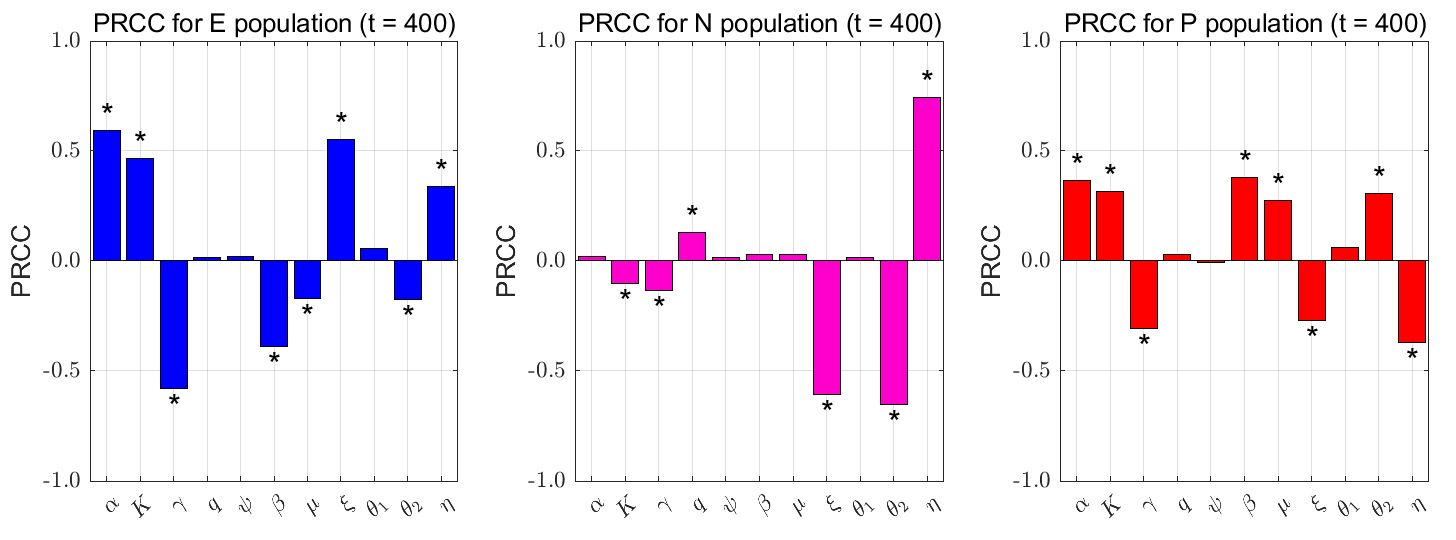}
    \caption{Figure shows the sensitivity of system \eqref{system_AMC-2025} parameters using the PRCC with t-statistics. Negative and positive bars indicate negative and positive correlations of system populations with respective parameters.}
    \label{fig_PRCC_sensitivity analysis}
\end{figure}

\begin{figure}[H]
    \centering
    \includegraphics[width=0.7\linewidth, height=0.4\linewidth]{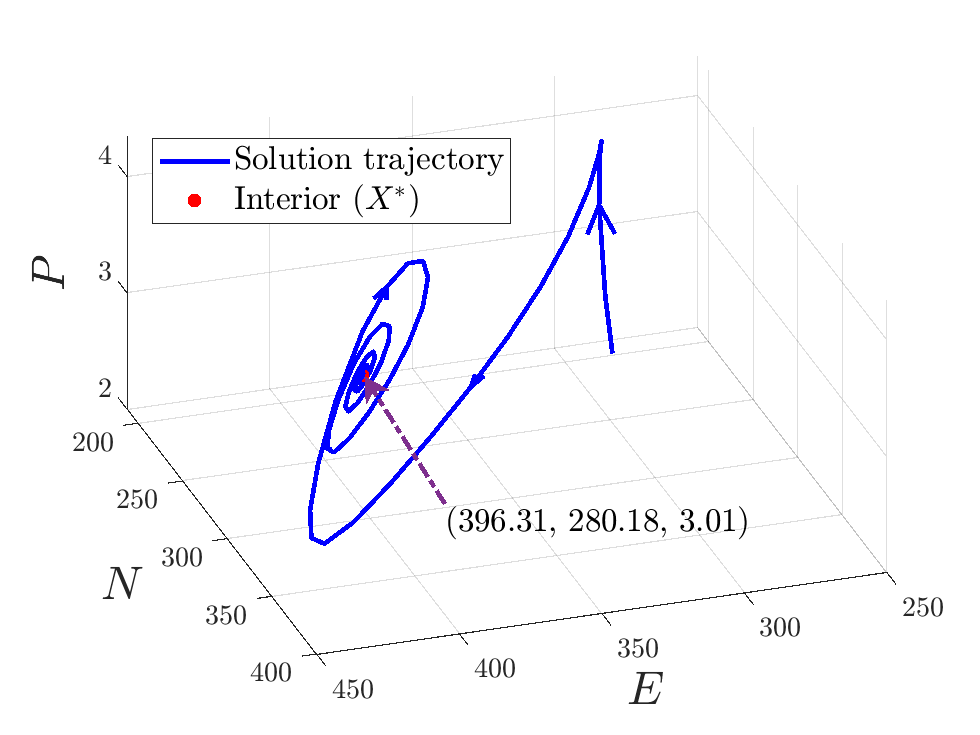}
    \caption{Figure interprets the solution trajectories of system \eqref{system_AMC-2025} are locally asymptotically stable to its coexistence equilibrium $X^*=\lx 396.31,\ 280.18,\ 3.01 \rx$ for parameter given in Table \ref{table_parameters}.}
    \label{fig_LAS-timeseries_coexistence}
\end{figure}

The Partial Rank Correlation Coefficient (PRCC) is employed to quantify the relative sensitivity of model outputs to variations in system parameters \cite{marino2008methodology}. In Figure \ref{fig_PRCC_sensitivity analysis}, we present the PRCC results for the parameters of system \eqref{system_AMC-2025}, computed around the baseline values specified in Table \ref{table_parameters}. Each parameter is varied within a $100\%$ interval of its baseline value. A total of $400$ time points are considered to capture the temporal influence of parameter perturbations on the population dynamics. The bar diagram illustrates the PRCC values, which lie within the interval $[-1,1]$, for each parameter in system \eqref{system_AMC-2025}. To assess statistical significance, the null hypothesis $H_0: \mathrm{PRCC}=0$ is tested individually for each parameter. The null hypothesis is rejected when $p<0.05$, indicating a statistically significant monotonic relationship between the parameter and the corresponding variable. Such significant parameters are denoted by a star $(*)$ in Figure \ref{fig_PRCC_sensitivity analysis}. The analysis reveals that all parameters, except $q$, $\psi$, and $\theta_1$, exert a statistically significant influence on the system populations.\\

\begin{figure}[]
    \centering
    \subfigure[$\g$ vs $\b$ biparametric stability region for the coexistence equilibrium $X^*$ of system \eqref{system_AMC-2025}.]{                         \includegraphics[width=0.47\linewidth]{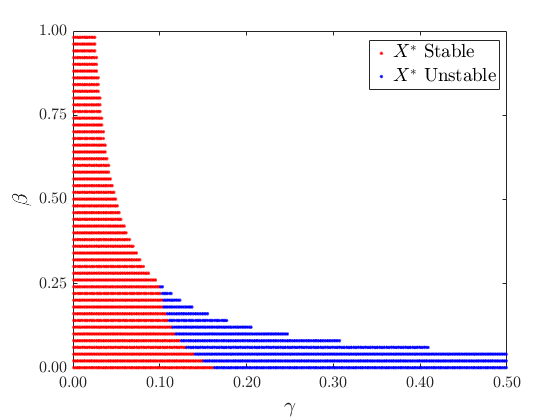}
        \label{subfig_stability region_gamma-beta}
    }
    \subfigure[$\g$ vs $\xi$ biparametric stability region for the coexistence equilibrium $X^*$ of system \eqref{system_AMC-2025}.]{                         \includegraphics[width=0.47\linewidth]{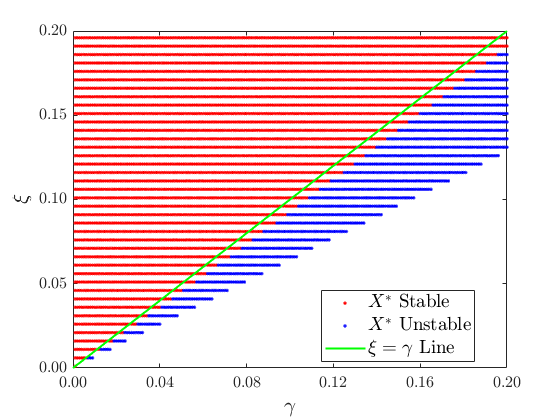}
        \label{subfig_stability region_gamma_xi}
    }
    \caption{Figure shows the stability region of the coexistence equilibrium $X^*$ of system \eqref{system_AMC-2025} in $(a) \g-\b$ and (b) $\g-\xi$ parametric planes. Red and blue dots in both subfigures represent the stable and unstable regions of $X^*$, respectively, whereas the white region indicates that $X^*$ does not exist. Other parameters are taken as in Table \ref{table_parameters}.}
    \label{fig_stability region_coexistence}
\end{figure} 
For the baseline parameter values provided in Table \ref{table_parameters}, system \eqref{system_AMC-2025} possesses three equilibria: the extinction equilibrium $X_0=(0.0,\ 0.0,\ 0.0)$, the equilibrium without the Banff townsite elk $X_1=(0.00,\ 300.00,\ 1.60)$, and the coexistence equilibrium $X^* = (396.31,\ 280.18,\ 3.01)$. Numerical simulations with initial condition $(E_0,\ N_0,\ P_0)=(340,\ 380,\ 4)$ show that the solution trajectories approach the coexistence equilibrium $X^*$, indicating that it is locally asymptotically stable for the given parameter set, as depicted in Figure \ref{fig_LAS-timeseries_coexistence}.\\

\begin{figure}[htp]
    \centering
    \subfigure[Bifurcation diagram of the Banff townsite elk population with respect to $\b$.]{\includegraphics[width=0.45\linewidth]{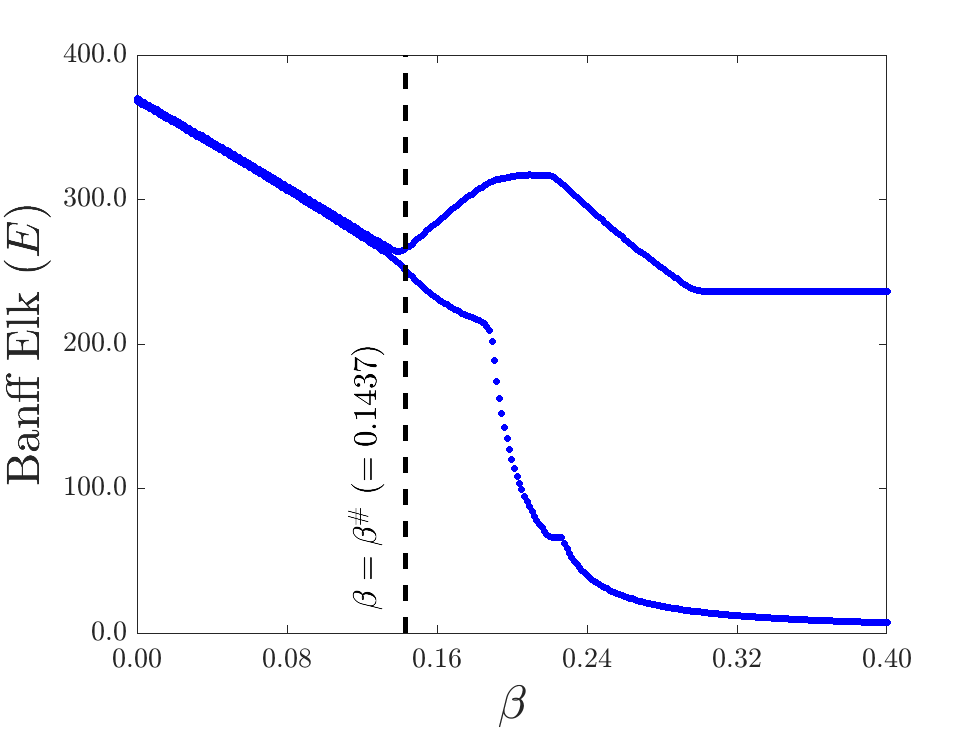}
        \label{subfig_bifurcation_E-beta}
    }\hspace{6pt}
    \subfigure[Bifurcation diagram of the Bow Valley elk population with respect to $\b$.]{\includegraphics[width=0.45\linewidth]{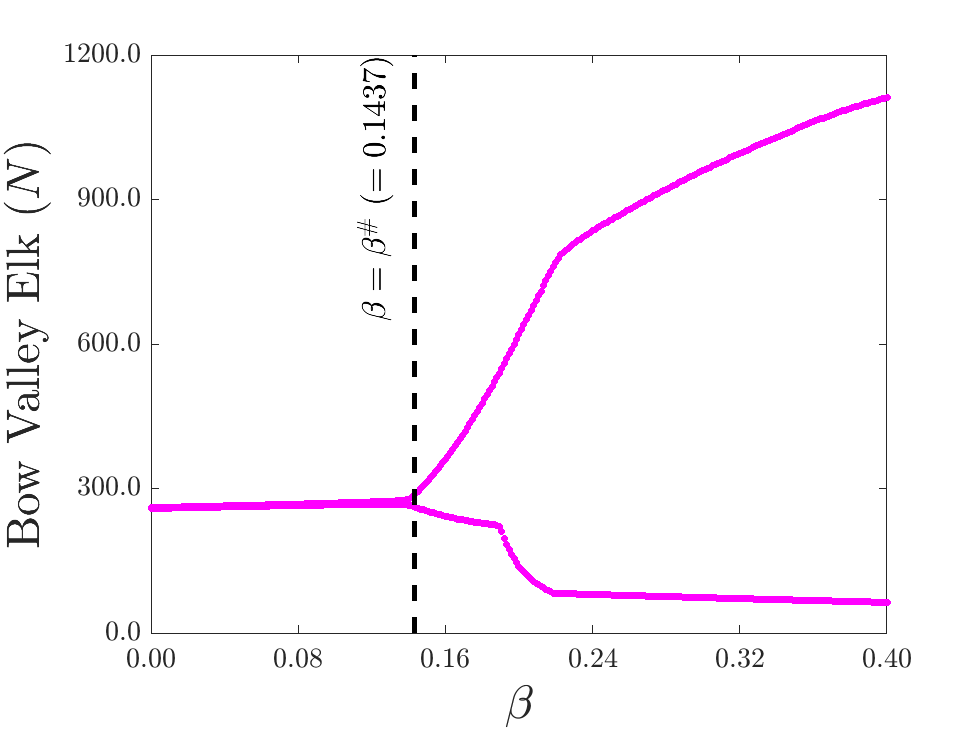}
        \label{subfig_bifurcation_N-beta}
    }
    \subfigure[Bifurcation diagram of the wolves population with respect to $\b$.]{\includegraphics[width=0.45\linewidth]{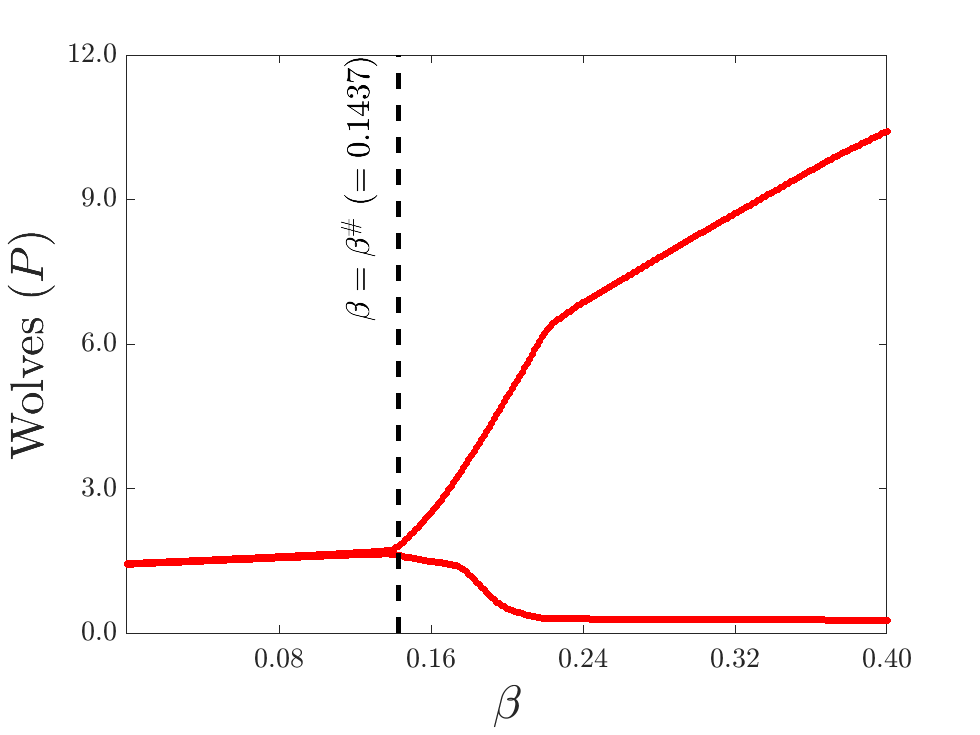}
        \label{subfig_bifurcation_P-beta}
    }
    \caption{Figure shows the bifurcation diagram of system \eqref{system_AMC-2025} with respect to the bifurcating parameter $\b$. It represents the occurrence of a Hopf bifurcation around the coexistence equilibrium at $\b=\b^\#\ (=0.1437)$ for the case $\g\ (=0.11)>\xi\ (=0.10)$. Other parameters are taken as in Table \ref{table_parameters}.}
    \label{fig_bifurcation_beta}
\end{figure}

The existence and stability regions of the coexistence equilibrium $X^*$ in selected two-parameter spaces are illustrated in Figure \ref{fig_stability region_coexistence}. In particular, subfigures \ref{subfig_stability region_gamma-beta} and \ref{subfig_stability region_gamma_xi} depict the existence and stability region of the coexistence equilibrium of system \eqref{system_AMC-2025} in $\g-\beta$ and $\g-\xi$ parameter planes, respectively, while all remaining parameters are fixed as listed in Table \ref{table_parameters}. In these diagrams, the red region corresponds to parameter combinations for which the coexistence equilibrium of system \eqref{system_AMC-2025} is locally asymptotically stable, whereas the blue region represents the existence of the coexistence when it is not stable. In the white region, the coexistence equilibrium does not exist. A notable feature emerging from these plots is that instability of the coexistence equilibrium occurs only when the interaction rate between Banff Townsite elk and wolves ($\g$) exceeds that between Bow Valley elk and wolves ($\xi$). Extensive numerical simulations further indicate that whenever $\g<\xi$, the coexistence equilibrium, provided it exists, remains locally asymptotically stable throughout the admissible parameter range. \\

\begin{figure}[htp]
    \centering
    \subfigure[Time evoltion of the Banff townsite elk population.]{\includegraphics[width=0.45\linewidth]{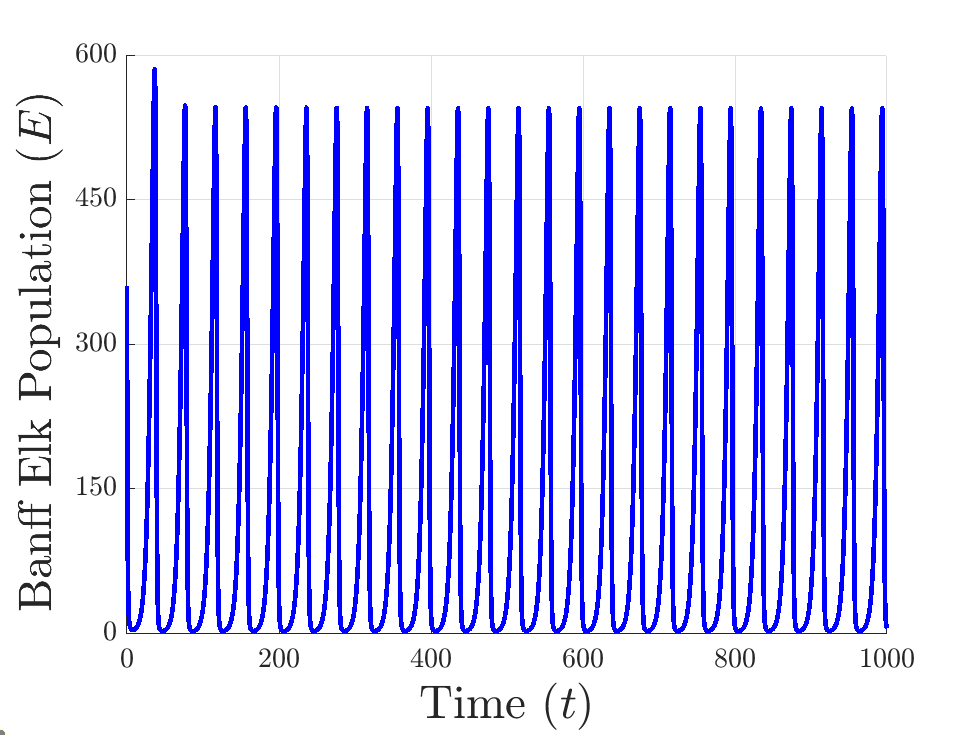}
        \label{subfig_Limit cycle_E}
    }\hspace{6pt}
    \subfigure[Time evoltion of the Bow Valley elk population.]{\includegraphics[width=0.45\linewidth]{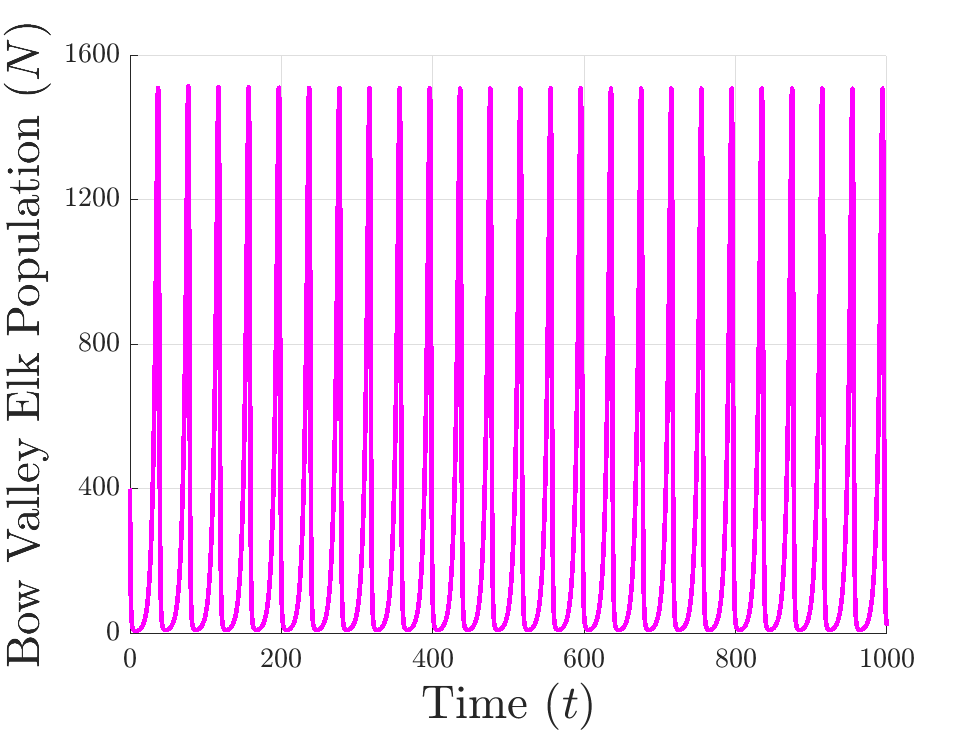}
        \label{subfig_Limit cycle_N}
    }
    \subfigure[Time evolution of the wolves population.]{\includegraphics[width=0.45\linewidth]{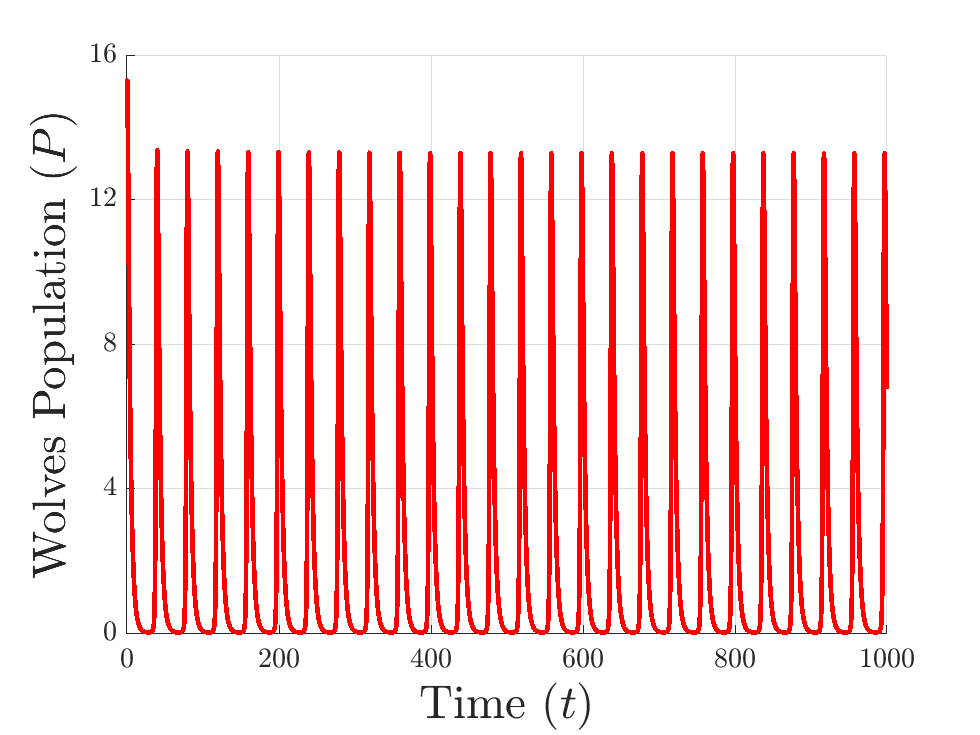}
        \label{subfig_Limit cycle_P}
    }\hspace{6pt}
    \subfigure[Phase portrait of system populations representing the limit cycle.]{\includegraphics[width=0.45\linewidth]{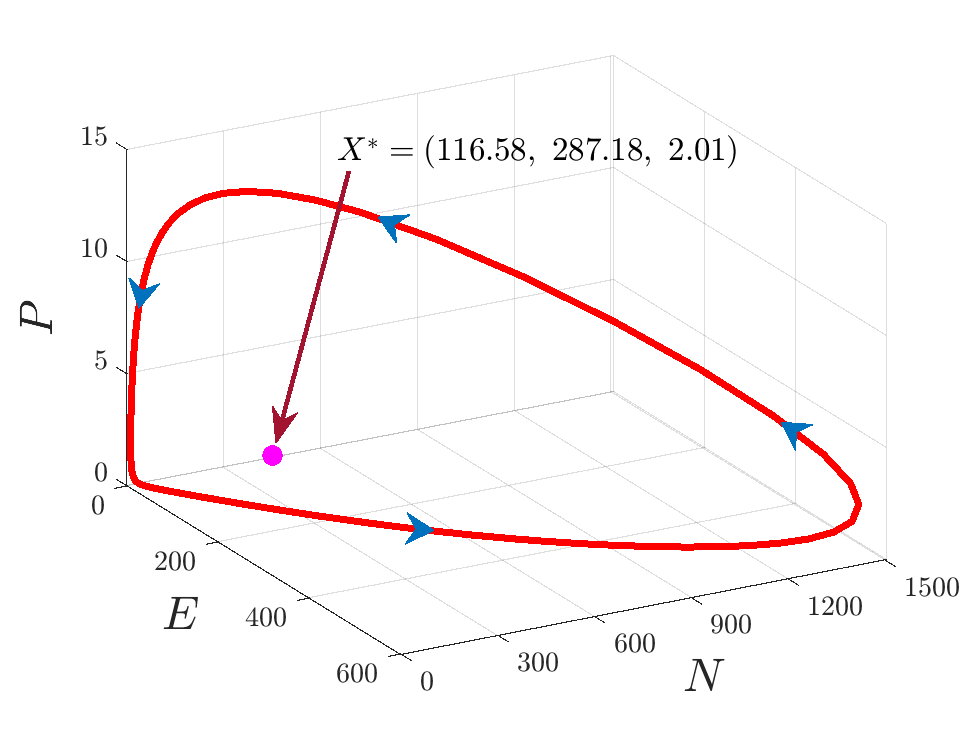}
        \label{subfig_Limit cycle_Phase portrait}
    }
    \caption{Figure shows the occurrence of a stable limit cycle around the coexistence equilibrium $(X^*)$ of system \eqref{system_AMC-2025} when $\g\ (=0.11)>\xi\ (=0.10)$ with other parameters as in Table \ref{table_parameters}.}
    \label{fig_Limit Cycle}
\end{figure}

For $\g=0.11$ (with $\g>\xi=0.10$) and the remaining parameters fixed at the baseline values given in Table \ref{table_parameters}, system \eqref{system_AMC-2025} undergoes a Hopf bifurcation at the coexistence equilibrium $X^*$ with respect to parameter $\b$. The corresponding bifurcation structure is illustrated in Figure \ref{fig_bifurcation_beta}. As $\b$ increases, the coexistence equilibrium $X^*$ loses its local asymptotic stability at the critical threshold $\b=\b^\#\ (=0.1437)$, leading to the emergence of a stable periodic orbit involving all three state variables. The bifurcation diagrams for the individual populations are presented in subfigures \ref{subfig_bifurcation_E-beta}, \ref{subfig_bifurcation_N-beta}, and \ref{subfig_bifurcation_P-beta}, which correspond to the Banff area elk, the Bow Valley elk, and the wolf populations, respectively, and clearly demonstrate the transition from steady-state coexistence to sustained oscillatory dynamics beyond the critical parameter value. \\

The time evolution and geometric structure of the periodic solution of system \eqref{system_AMC-2025}, with initial conditions $(E_0,\ N_0,\ P_0)=(360,\ 400,\ 15)$, corresponding to the stable limit cycle, are presented in Figure \ref{fig_Limit Cycle}. Subfigures \ref{subfig_Limit cycle_E}, \ref{subfig_Limit cycle_N}, and \ref{subfig_Limit cycle_P} display the temporal dynamics of the Banff townsite elk, the Bow Valley elk, and the wolf populations, respectively, highlighting sustained oscillations in each compartment. Furthermore, subfigure \ref{subfig_Limit cycle_Phase portrait} illustrates the associated phase portrait, which confirms the occurrence of a limit cycle encircling the coexistence equilibrium $X^*$, thereby demonstrating persistent periodic coexistence of all three populations.

\section{Conclusions}\label{sec_conclusions}

In this note, we have undertaken a detailed mathematical reassessment of the elk–wolf prey–predator model with inter-regional movements between refuge (Banff townsite area) and open habitat (Bow Valley region) proposed in \cite{maji2026persistence}. Through a systematic re-derivation of the analytical results, we clarified the conditions ensuring positivity, boundedness, and feasibility of equilibria. In particular, we established the precise parameter restrictions under which the system remains biologically meaningful and demonstrated that certain previously used parameter sets do not satisfy these boundedness requirements.\\

The existence and local stability conditions for all the equilibria of system \eqref{system_AMC-2025} were rigorously re-examined. Explicit expressions for the coexistence equilibrium were carefully derived, and the associated Routh–Hurwitz conditions were simplified to some extent. We further revisited the global stability analysis and identified necessary corrections in the construction and application of the Lyapunov function. By providing appropriate definiteness conditions and clarifying the required inequalities, we presented a mathematically sound framework for assessing global asymptotic stability of the coexistence equilibrium.\\

The bifurcation analysis was also reconsidered in detail. We re-derived the Hopf bifurcation conditions using the proper transversality criterion and obtained explicit expressions for the derivative of the critical eigenvalues with respect to the bifurcation parameter under the condition $\g>\xi$. Furthermore, using normal form theory, we outlined the correct procedure for computing the first Lyapunov coefficient and determining the direction, stability, and period variation of the bifurcating limit cycle. Our analysis highlights the importance of distinguishing between genuine Hopf bifurcation phenomena and dynamics arising from reduced subsystem behavior, such as Lotka–Volterra oscillations following species extinction.\\

Extensive numerical simulations were performed to validate the corrected theoretical results. Sensitivity analysis via PRCC, biparametric stability diagrams, bifurcation plots, and phase portraits were recalculated using parameter sets consistent with the analytical constraints. The numerical results confirm the revised stability and bifurcation structure and demonstrate the conditions under which sustained oscillations emerge.\\

Overall, this study provides mathematically consistent corrections and clarifications to the previously reported results. By refining the analytical framework and ensuring reproducibility of numerical findings, the present work contributes to a more rigorous understanding of refuge-mediated predator–prey systems and offers a reliable foundation for future investigations in ecological modeling and bifurcation analysis.

\section*{Acknowledgements}
RD acknowledges the financial support from the Ministry of Education (MoE), Govt of India.

\section*{Conflict of Interest}
The authors declare that they have no conflict of interest in the present study.

\section*{Data Availability}
The data used in this article are included in this article only.


\bibliographystyle{elsarticle-num}
\bibliography{Rajesh}
\end{document}